\documentclass[reqno]{amsart}
\usepackage{amsmath,amsthm,amsfonts}
\newtheorem{thm}{Theorem}
\newtheorem{lem}[thm]{Lemma}
\newtheorem{cor}[thm]{Corollary}
\theoremstyle{definition}

\newtheorem{rem}[thm]{Remark}

\newcommand{\Z}{\mathbb{Z}}

\newcommand{\nexteq}{\displaybreak[0]\\ &=}
\newcommand{\allone}{\mathbf{1}}
\DeclareMathOperator{\wt}{wt}
\DeclareMathOperator{\diag}{diag}

\DeclareMathOperator{\Lee}{Lee}
\DeclareMathOperator{\Norm}{Norm}
\newcommand{\NormI}{\Norm_{\text{I}}}
\newcommand{\NormII}{\Norm_{\text{II}}}
\newcommand{\Norme}{\Norm_{\text{e}}}
\newcommand{\Normo}{\Norm_{\text{o}}}

\begin{document}

\title{The Codes and the Lattices
of Hadamard Matrices}
\author[A. Munemasa]{Akihiro Munemasa}
\address{Graduate School of Information Sciences,
Tohoku University, Sendai, 980-8579 Japan}
\email{munemasa@math.is.tohoku.ac.jp}

\author[H. Tamura]{Hiroki Tamura}
\address{Graduate School of Information Sciences,
Tohoku University, Sendai, 980-8579 Japan}
\email{tamura@ims.is.tohoku.ac.jp}

\subjclass[2000]{Primary 05B20; Secondary 94C30, 11H71}

\keywords{Hadamard matrix,
self-dual code, Leech lattice, even unimodular lattice}

\date{October 22, 2011}

\begin{abstract}
It has been observed by Assmus and Key
as a result of the
complete classification of Hadamard matrices of order $24$,
that the extremality of the binary
code of a Hadamard matrix $H$ of order $24$
is equivalent to the extremality of the
ternary code of $H^T$.
In this note, we 
present two proofs of this fact, neither of which
depends on the classification. One
is a consequence of a more general
result on the minimum weight of the dual of the code
of a Hadamard matrix. The other relates the lattices
obtained from the binary code and from the ternary 
code.
Both proofs are presented in greater generality to include
higher orders. In particular, 
the latter method is also used to show the equivalence
of (i) the extremality of the ternary code, (ii) the 
extremality of the $\Z_4$-code, and (iii) the extremality of
a lattice obtained from a Hadamard matrix of order $48$.
\end{abstract}

\maketitle

\section{Introduction}
A Hadamard matrix is a square matrix $H$ of order $n$ 
with entries $\pm1$ satisfying $HH^T=nI$, where $I$ denotes
the identity matrix.
If $m$ is an odd
integer such that $n\equiv0\pmod{m}$ and $(m,\frac{n}{m})
=1$, then the row vectors of a Hadamard matrix
of order $n$ generate a
self-dual code of length $n$ over $\Z/m\Z$, called the code
of $H$ over $\Z/m\Z$.
In particular, the ternary code of a
Hadamard matrix of order $24$ is a 
self-dual code of length $24$. A ternary self-dual 
code of length $24$ 
is called extremal if its minimum weight is
$9$. Such codes have been classified in \cite{LPS}, 
and there are
exactly two extremal ternary self-dual codes of length
$24$, up to equivalence. It is known that,
from the classification of Hadamard matrices of order $24$
(see \cite{ILu,IL,K}),
there are exactly two Hadamard matrices, up to equivalence,
whose codes are extremal ternary self-dual codes. One is
the Paley matrix, and the other is the matrix $H58$
(cf \cite{AK}).

For a Hadamard matrix $H$, the matrix 
$B=\frac12(H+J)$, where $J$ denotes the all-one matrix, is called the
binary Hadamard matrix associated to $H$. 
A Hadamard matrix $H$ is said to be normalized if all
the entries of its first row are $1$. 
For a normalized Hadamard matrix $H$,
the binary code generated by the row vectors of
the binary Hadamard matrix associated to $H$ is called the binary
code of $H$. It is not difficult to check that if
$H,H'$ are Hadamard
equivalent normalized Hadamard matrices, then
the binary codes of $H,H'$ are equivalent. The binary
code of a Hadamard matrix of order $n$ is doubly even self-dual
if $n\equiv8\pmod{16}$ (see \cite[Section 17.3]{Hall}). 
More generally, the code over $\Z/2m\Z$ generated by
the row vectors of $B$ is type II self-dual if
$n\equiv0\pmod{8m}$ and $(2m,\frac{n}{8m})=1$.
In particular, the binary code of
every normalized Hadamard matrix of order $24$ is a binary
doubly even self-dual code of length $24$. 
A binary doubly even self-dual code of length $24$
is called extremal if its minimum weight is $8$.
The extended binary Golay 
code is the unique extremal
binary doubly even self-dual code length $24$.
It is known that,
from the classification of Hadamard matrices of order $24$,
there are exactly two normalized 
Hadamard matrices, up to equivalence,
whose binary codes are equivalent to the extended binary
Golay code. One is
the Paley matrix, and the other is the matrix $H8$
(cf \cite{AK}).

Among the sixty equivalence classes of Hadamard matrices
of order $24$, only two correspond to extremal ternary
self-dual codes, and also only two correspond to extremal
binary doubly even self-dual codes. Somewhat remarkable
fact \cite[p.~286]{AK} was that, apart from the Paley matrix
which is common to the ternary and the binary cases,
the transpose of the Hadamard
matrix $H58$ is Hadamard equivalent to the matrix $H8$. 
Since the Paley matrix is Hadamard 
equivalent to its transpose,
this phenomenon makes one wonder if there is any reason
why the extremality of the ternary code of a Hadamard
matrix is equivalent to the extremality of the binary
code of its transpose.
The purpose of this paper is to give a
theoretical explanation of this phenomenon, which does
not depend on the classification of Hadamard matrices
of order $24$. Two different proofs will be given of this
fact. In Section~3, we give an elementary and direct
method to analyze the existence of a codeword of small
Euclidean norm 
 in the dual of the code of a Hadamard matrix. This
method can be adapted to deal with the binary case, and
the proof is a simple consequence (Corollary~\ref{cor:1}).
In Section~4, we will consider the unimodular lattices
obtained from the $\Z_m$-code and the $\Z_{n/4m}$-code
of a (binary) Hadamard matrix of order $n$.
It is shown in particular, that the lattice obtained
from the ternary code of a Hadamard matrix $H$ of order $24$
is isometric to a neighbor $L$ of the lattice $L_2$
obtained from the binary code of $H$. Then the extremality
of the ternary code or that of the binary code is shown
to be equivalent to the common neighbor $\Lambda$ of $L$ and $L_2$
being the Leech lattice.
We also show that the extremality of the ternary code of
a Hadamard matrix of order $48$ is equivalent to the
extremality (in the sense of Euclidean norm) of the
$\Z_4$-code of its binary transpose,
and to the extremality of the even unimodular lattice
obtained as above $\Lambda$. 
We note that a weaker equivalence for order $48$ will be
proved in Section~3 without using lattices.

\section{Elementary divisors of Hadamard matrices}

We denote the all-one matrix by $J$, and the all-one vector
by $\allone$. We also denote by $e_i$
the vector with a $1$ in the $i$-th coordinate and $0$ elsewhere.
We refer the reader to
\cite{MS} for unexplained terminology in codes.

\begin{lem}\label{lem:gcd}
If positive integers $x,y,z,w$ satisfy $xy=wz$ and $(x,y)=1$,
then $x=(x,z)(x,w)$.
\end{lem}
The following lemma follows immediately from \cite[Chap. II, Exercise 4]{Newman}.
See also \cite[Part 4, Theorem 10.7]{WSW}.
\begin{lem}\label{lem:divH}
Let $H$ be a Hadamard matrix of order $n$, and let
$d_1|d_2|\cdots|d_n$ be the elementary divisors of $H$.
Then we have $d_id_{n+1-i}=n$ for all $i$.
\end{lem}
\begin{proof}
Take $P,Q\in GL(n,\Z)$ so that $PHQ=\diag(d_1,\dots,d_n)$.
Then we have $Q^{-1}H^TP^{-1}=\diag(\frac{n}{d_1},\dots,\frac{n}{d_n})$ and
$\frac{n}{d_n}|\cdots|\frac{n}{d_2}|\frac{n}{d_1}$ are also
the elementary divisors of $H$.
\end{proof}
\begin{lem}\label{lem:sd}
Let $H$ be a Hadamard matrix of order $n$, $m$ an integer
such that $m|n$ and $(m,\frac{n}{m})=1$.
Then the row vectors of $H$ generate
a self-dual code of length $n$ over $\Z/m\Z$.
\end{lem}
\begin{proof}
Let $C$ be the code over $\Z/m\Z$ generated by the row vectors
of $H$. Clearly, $C$ is self-orthogonal.
Let $d_1|d_2|\cdots|d_n$ be the elementary divisors of $H$. Since
\begin{align*}
|C|&=\prod_{i=1}^n\frac{m}{(m,d_i)}\\
&=\prod_{i=1}^{n/2}\frac{m}{(m,d_i)}\cdot\frac{m}{(m,n/d_i)}&&\text{(by Lemma~\ref{lem:divH})}\\
&=\prod_{i=1}^{n/2}\frac{m^2}{m}&&\text{(by Lemma~\ref{lem:gcd})}\\
&=m^{n/2},
\end{align*}
$C$ is self-dual.
\end{proof}
%
%
\begin{lem}\label{lem:divB}
Let $H$ be a Hadamard matrix of order $n$, normalized
in such a way that the entries of its first row are all
$1$. Let $B$ be the binary Hadamard matrix associated to $H$.
If the elementary divisors of $H$ are $1=d_1|d_2|\cdots|d_n$,
then those of $B$ are $1|\frac{d_2}{2}|\cdots|\frac{d_n}{2}$.
\end{lem}
\begin{proof}
We can assume that $H$ is normalized as 
$\begin{pmatrix}1&\begin{matrix}\dots&1\end{matrix}\\ \begin{matrix}-1\\ \vdots\\-1\end{matrix}&H'\end{pmatrix}$.
Then
\begin{align*}
\begin{pmatrix}1&0\dots 0\\\begin{matrix}1\\ \vdots\\ 1\end{matrix}&I_{n-1}\end{pmatrix}H&=
\begin{pmatrix}1&\begin{matrix}\dots&1\end{matrix}\\\begin{matrix}0\\ \vdots\\0\end{matrix}&H'+J\end{pmatrix}\\
\intertext{and}
B=\frac{1}{2}(H+J)&=
\begin{pmatrix}1&\begin{matrix}\dots&1\end{matrix}\\\begin{matrix}0\\ \vdots\\0\end{matrix}&\frac{1}{2}(H'+J)\end{pmatrix}.
\end{align*}
The result follows by comparing the above two equalities.
\end{proof}
%
Let $m$ be a positive integer, and set $V=\Z/m\Z$.
We regard an element $u\in V$ as an element of the set
of integers $\{0,1,\dots,m-1\}$,
and define the Lee weight and the Euclidean norm of 
an element $u\in V$ by
\begin{align*}
\Lee(u)&=\min\{u,m-u\},\\
\Norm(u)&=(\Lee(u))^2.
\end{align*}
For a vector $u=(u_1,\dots,u_n)\in V^n$, we set
\[
\Norm(u)=\sum_{i=1}^n\Norm(u_i).
\]
Alternatively, the Euclidean norm can be defined as
\[
\Norm(u)=\min\left\{\|v\|^2\mid v\in\Z^n,v\bmod m=u\right\}.
\]

Recall that a self-dual code over $\Z/2m\Z$ is type II if
the Euclidean norm of every codeword is divisible by $4m$.

\begin{lem}\label{lem:sd2}
Let $H$ be a normalized Hadamard matrix of order $n$,
$B$ the binary Hadamard matrix associated to $H$.
Let $\ell\ge2$ be an integer such that $4\ell|n$ and $(\ell,\frac{n}{4\ell})=1$.
Then the row vectors of $B$ generate
a self-dual code over $\Z/\ell\Z$ of length $n$, 
which is type II if $\ell$ is even.
\end{lem}
\begin{proof}
Let $C$ be the code over $\Z/\ell\Z$ generated by the row vectors of $B$.
Since $H$ is normalized, we have
\begin{align}
BB^T&=\frac{1}{4}(H+J)(H^T+J)\nonumber\\
&=\frac{n}{4}(I+e_1^T\allone+\allone^Te_1+J)\label{eq:orth}\\
&\equiv0\pmod{\ell}.\nonumber
\end{align}
Thus $C$ is self-orthogonal.
Let $d_1|d_2|\cdots|d_n$ be the elementary divisors of $H$.
Since 
\[
(\ell,d_i/2)(\ell,n/2d_i)=(\ell,d_i/2)(\ell,(n/4)/(d_i/2))=\ell
\]
by Lemma~\ref{lem:gcd}, we have
\begin{align*}
|C|&=\ell\prod_{i=2}^n\frac{\ell}{(\ell,d_i/2)}&&\text{(by Lemma~\ref{lem:divB})}\\
&=\ell\prod_{i=2}^{n/2}\frac{\ell}{(\ell,d_i/2)}\cdot\frac{\ell}{(\ell,n/2d_i)}
&&\text{(by Lemma~\ref{lem:divH})}\\
&=\ell\prod_{i=2}^{n/2}\frac{\ell^2}{\ell}\\
&=\ell^{n/2}.
\end{align*}
Thus $C$ is self-dual.
Finally, since $\ell|\frac{n}{4}$, (\ref{eq:orth}) implies that 
the diagonal entries of $BB^T$ are divisible by $2\ell$. Thus
$C$ is type II if $\ell$ is even by \cite[Lemma 2.2]{BDHO}.
\end{proof}

\section{Minimum weights of codes of Hadamard matrices}
We introduce two types of pair of norms of a vector over $V=\Z/m\Z$.
First, assume $m$ is odd. We define the odd norm and the even norm by
\begin{align*}
\Normo(u)
=&\min\left(\left\{\|v\|^2\mid v\in\Z^n,\;v\bmod m=u\right\}
\cap (1+2\Z)\right),\\
\Norme(u)
=&\min\left(\left\{\|v\|^2\mid v\in\Z^n,\;v\bmod m=u\right\}
\cap 2\Z\right).
\end{align*}
The assumption that $m$ be odd is required to ensure that
both parities occur among the norms of vectors $v$ satisfying
$v\bmod m=u$.
If $u=v\bmod m$ and $\Norm(u)=\|v\|^2$, then
\begin{align}\label{eq:n}
\{\Normo(u),\Norme(u)\}
&=\{\|v\|^2,\min_i\{\|v\pm me_i\|^2\}\}\nonumber\\
&=\{\Norm(u),\Norm(u)+m(m-2\max_i\{\Lee(u_i)\})\}.
\end{align}
In particular, for $u\neq0$, we have
\[
|\Normo(u)-\Norme(u)|\leq m^2-2m.
\]

\begin{lem}\label{lem:norm_m}
Let $H$ be a Hadamard matrix of order $n$, and
let $C$ be the code over $\Z/m\Z$ generated by the rows of $H$,
where $m\geq 3$ is an odd integer.
Then the following statements hold.
\begin{itemize}
\item[{\rm(i)}]
$C^\perp$ has no codeword of odd norm less than $m^2$,
\item[{\rm(ii)}]
For any $u\in C^\perp\setminus\{0\}$, we have $\Norm(u)\geq 2m$.
Equality holds only if nonzero entries of $u$
are all equal to $1$ or $-1$.
\end{itemize}
\end{lem}
\begin{proof}
(i) Let $v$ be a vector in $\Z^n$ such that $v\bmod m$ is
$u\in C^\perp$ and $\|v\|^2=\Normo(u)$.
Then we have $vH^T\equiv 0\pmod m$ and
$vH^T\equiv v\allone^T\allone\equiv\allone\pmod 2$
and thus $vH^T\equiv m\allone\pmod{2m}$.
So we have $\|v\|^2=\frac{1}{n}vH^THv^T=\frac{1}{n}\|vH^T\|^2\geq m^2$.

(ii) By (i) and (\ref{eq:n}), we have
\[
m^2\leq\Normo(u)\leq\Norm(u)+m(m-2\max_i\{\Lee(u_i)\}).
\]
So we have $1\leq\max_i\{\Lee(u_i)\}\leq\frac{\Norm(u)}{2m}$.
\end{proof}

Next,
we define type I norm and type II norm for an integer $m$ and 
$u\in\allone^\perp\subset V^n$ by
\begin{align*}
\NormI(u)
=&\min\left\{\|v\|^2\mid v\in\Z^n,\;v\bmod m=u,\;
v\cdot\allone\equiv m\!\pmod{2m}\right\},\\
\NormII(u)
=&\min\left\{\|v\|^2\mid v\in\Z^n,\;v\bmod m=u,\;
v\cdot\allone\equiv 0\pmod{2m}\right\}.
\end{align*}
If $u=v\bmod m$ and $\Norm(u)=\|v\|^2$, then
\begin{align}\label{eq:n'}
\{\NormI(u),\NormII(u)\}
&=\{\|v\|^2,\min_i\{\|v\pm me_i\|^2\}\}\nonumber\\
&=\{\Norm(u),\Norm(u)+m(m-2\max_i\{\Lee(u_i)\})\}.\tag{2'}
\end{align}
In particular, for $u\neq0$, we have
\[
|\NormI(u)-\NormII(u)|\leq m^2-2m.
\]

\begin{lem}\label{lem:norm_l}
Let $H$ be a normalized Hadamard matrix of order $n$, and 
let $B$ be the binary Hadamard matrix associated to $H$.
Let $C$ be the code over $\Z/\ell\Z$ generated by the rows of $B$,
where $\ell\geq 2$ is an integer.
Then the following statements hold.
\begin{itemize}
\item[{\rm(i)}]
$C^\perp$ has no codeword of type I norm less than $\ell^2$,
\item[{\rm(ii)}]
For any $u\in C^\perp\setminus\{0\}$, we have $\Norm(u)\geq 2\ell$.
Equality holds only if nonzero entries of $u$
are all equal to $1$ or $-1$.
\end{itemize}
\end{lem}
\begin{proof}
(i) Let $v$ be a vector in $\Z^n$ such that $v\bmod \ell$ is
$u\in C^\perp$, $v\cdot\allone\equiv \ell\pmod{2\ell}$
and $\|v\|^2=\NormI(u)$.
Then we have $vB^T\equiv 0\pmod \ell$ and thus
$vH^T=v(2B^T-J)=2vB^T-(v\cdot\allone)\allone\equiv \ell\allone\pmod{2\ell}$.
So we have $\|v\|^2=\frac{1}{n}\|vH^T\|^2\geq \ell^2$.

(ii) The proof is similar to that of Lemma~\ref{lem:norm_m} (ii).
\end{proof}

When $m$ is odd, and $u\in \allone^\perp$, then
\begin{align*}
\NormI(u)&=\Normo(u),\\
\NormII(u)&=\Norme(u).
\end{align*}
The Euclidean norm over $\Z/2\Z$ or $\Z/3\Z$ is equal to the weight.
Moreover both type I norm and type II norm over $\Z/2\Z$ 
are equal to the weight.

The minimum odd, even, type I, and type II norms of a code $C$ 
over $\Z/m\Z$ are 
defined by
$$\min(\{\Norm_*(u)\mid u\in C\}\setminus\{0\}),\qquad 
*=\text{o, e, I, II,}$$
respectively, provided $m$ is odd for $*=\text{o,e}$, 
$C\subset\allone^\perp$ for $*=\text{I,II}$.
Note that the minimum odd norm and the minimum type I norm of a code over $\Z/m\Z$
is at most $m^2$ which is the odd norm and the type I norm of the zero vector.

\begin{thm}\label{thm:C2Cp'}
Let $H$ be a normalized Hadamard matrix of order $n$, and let $B$
be the binary Hadamard matrix associated to $H$.
Let $m\geq 3$ be an odd integer, and let $\ell\geq 2$ be an integer
satisfying $(\ell,m)=1$ and $n\equiv0\pmod{4\ell m}$.
Let $C_m$ be the code over $\Z/m\Z$ generated by the rows of
$H^T$,
and let $C'_\ell$ be the code over $\Z/\ell\Z$ generated by the rows of $B$.
Then the following statements hold.
\begin{itemize}
\item[{\rm(i)}]
Suppose $C_m^\perp$ has a codeword of even norm $d$ and odd norm $m^2+k$ where $k<d-d/\ell$.
Then there exists a vector $v\in\Z^n$ such that $u=(1/2m)vH\bmod \ell$ is a nonzero codeword
of $C'_\ell$ with  $\NormII(u)\leq dn/4m^2$.
If, moreover $d<2m\lfloor(\ell+2)/2\rfloor$, then $\NormII(u)=\Norm(u)=dn/4m^2$,
and if $k=0$, then $\NormII(u)=\Norm(u)=\wt(u)=dn/4m^2$.
\item[{\rm(ii)}]
Suppose ${C'_\ell}^\perp$ has a codeword of type II norm $d$ and type I norm $\ell^2+k$ where $k<d-d/m$.
Then there exists a vector $v\in\Z^n$ such that $u=(1/2\ell)vH^T\bmod m$ is a nonzero codeword
of $C_m$ with  $\Norme(u)\leq dn/4\ell^2$.
If, moreover $d<\ell(m+1)$, then $\Norme(u)=\Norm(u)=dn/4\ell^2$,
and if $k=0$, then $\Norme(u)=\Norm(u)=\wt(u)=dn/4\ell^2$.
\end{itemize}
\end{thm}
\begin{proof}
(i) By the assumption, there exists
a vector $v\in\Z^n$ satisfying $vH\equiv0\pmod{2m}$,
$\|v\|^2=d$ and $\|v-me_i\|^2=m^2+k$ for some $i$.
Let $c=(c_1,\dots,c_n)=\frac{1}{2m}vH$. We will show that $c\bmod \ell$ is a nonzero codeword of $C'_\ell$
with the desired property.
Since $(\ell,m)=1$, there exists an integer $t$ such that
$mt\equiv1\pmod \ell$, and
$c=\frac{1}{m}(vB-\frac{1}{2}vJ)\equiv t(vB-\frac{v\cdot\allone}{2}\allone)\pmod \ell$.
Thus $c\bmod \ell$ is a codeword of $C'_\ell$,
and since $c\allone^T=\frac{n}{2m}v_1\equiv 0\pmod{2\ell}$, we have 
$$\NormII(c\bmod \ell)\leq\|c\|^2=\frac{1}{(2m)^2}vHH^Tv^t=\frac{dn}{4m^2}.$$
We show $c\bmod \ell\ne 0$.
We have $(v-me_i)H=m(2c_1-h_{i1},\dots,2c_n-h_{in})$ where $H=\left(h_{ij}\right)_{i,j}$, and
\begin{align*}
dn=\|vH\|^2&=4m^2\sum_{j=1}^nc_j^2,\\
(m^2+k)n=\|(v-me_i)H\|^2&=m^2\sum_{j=1}^n(2c_j-h_{ij})^2\\
&=m^2n+4m^2\sum_{j=1}^nc_j(c_j-h_{ij}).
\end{align*}
Since $k\ell<d(\ell-1)$, we have
\begin{align*}
\sum_{j=1}^n|c_j|(|c_j|-\ell)&=\sum_{j=1}^n(c_j^2-\ell|c_j|)\\
&\leq\sum_{j=1}^n(c_j^2-\ell h_{ij}c_j)
\nexteq\sum_{j=1}^n(\ell c_j(c_j-h_{ij})-(\ell-1)c_j^2)
\nexteq\frac{k\ell n}{4m^2}-\frac{d(\ell-1)n}{4m^2}
\nexteq\frac{(k\ell-d(\ell-1))n}{4m^2}\\
&<0.
\end{align*}
Thus there exists $j\in\{1,\dots,n\}$ such that $0<|c_j|<\ell$.
Therefore, $c\bmod \ell\ne 0$.

It remains to show $\Norm(c\bmod \ell)=\NormII(c\bmod \ell)=\|c\|^2$
when $d<2m\lfloor\frac{\ell+2}{2}\rfloor$ or $k=0$.
This will follow if $|c_j|\leq\lfloor\frac{\ell}{2}\rfloor$ for all $j$.
If $d<2m\lfloor\frac{\ell+2}{2}\rfloor$, we have
$|c_j|=|(\frac{1}{2m}vH)_j|\leq\lfloor\frac{\|v\|^2}{2m}\rfloor
\leq\frac{d}{2m}<\lfloor\frac{\ell+2}{2}\rfloor$
and thus $|c_j|\leq\lfloor\frac{\ell}{2}\rfloor$.
If $k=0$, we have $c_j(c_j-h_{ij})=0$ and thus $c_j\in\{0,\pm1\}$,
which implies $\wt(c\bmod \ell)=\|c\|^2$.

(ii) If ${C'_\ell}^\perp$ has such a codeword, then there exists
a vector $v\in\Z^n$ satisfying
$vB^T\equiv0\pmod \ell$, $v\cdot\allone\equiv0\pmod{2\ell}$,
$\|v\|^2=d$ and $\|v-\ell e_i\|^2=\ell^2+k$ for some $i$.
Since $H^T=2B^T-J$, we have
$vH^T\equiv -(v\cdot\allone)\allone\equiv0\pmod{2\ell}$.
Let $c=(c_1,\dots,c_n)=\frac{1}{2\ell}vH^T$.
Since $(2\ell,m)=1$, there exists an integer $t$ such that
$2\ell t\equiv1\pmod m$, and $c\equiv tvH^T\pmod m$.
Thus $c\bmod m$ is a codeword of $C_m$,
and since $\|c\|^2\equiv ch_1^T=\frac{n}{2\ell}v_1\equiv0\pmod2$
where $h_1^T$ is the first column of $H$, we have 
$$\Norme(c\bmod m)\leq\|c\|^2=\frac{1}{(2m)^2}vH^THv^t=\frac{dn}{4m^2}.$$
By the same argument as above, we have
$c\not\equiv0\pmod m$ for $k<d-\frac{d}{m}$.
In particular, for all $j$, we have $|c_j|<\frac{m}{2}$ if $d<\ell(m+1)$,
and $c_j\in\{0,\pm1\}$ if $k=0$.
\end{proof}
\begin{cor}
\label{cor:C2Cp'}
Under the same assumption and notation as in Theorem~\ref{thm:C2Cp'},
the following hold for $0<d<2\ell m$.
\begin{itemize}
\item[{\rm(i)}]
If $C_m^\perp$ has a codeword of even norm $d$,
then $C'_\ell$ has a nonzero codeword of type II norm at most $dn/4m^2$.
\item[{\rm(ii)}]
If ${C'_\ell}^\perp$ has a codeword of type II norm $d$,
then $C_m$ has a nonzero codeword of even norm at most $dn/4\ell^2$.
\end{itemize}
\end{cor}

A ternary self-dual $[n,n/2]$ code $C$ has minimum weight
at most $3\lfloor n/12\rfloor+3$, and $C$ is called extremal
if $C$ has minimum weight exactly $3\lfloor n/12\rfloor+3$.
For $n=24$, extremal ternary self-dual codes are those
self-dual codes having no codewords of weight $3$ or $6$.
It is known that
there are two extremal ternary self-dual codes of length
$24$ up to equivalence (see \cite{LPS}).

A binary doubly even self-dual $[n,n/2]$ code $C$ has
minimum weight at most $4\lfloor n/24\rfloor+4$, and $C$
is called extremal if $C$ has
minimum weight exactly $4\lfloor n/24\rfloor+4$.
For $n=24$, 
extremal binary doubly even self-dual codes are those
binary doubly even self-dual codes having no codewords of
weight $4$. It is known that there is a
unique extremal binary doubly even self-dual
code of length $24$ up to equivalence, namely,
the extended binary Golay code.

\begin{cor}\label{cor:1}
Let $H$ be a normalized Hadamard matrix of order $24$.
Let $C_3$ be the ternary code generated by the rows of
$H^T$, and let $C'_2$ be the binary code of $H$.
Then $C_3$ is an extremal self-dual
$[24,12,9]$ code if and only if $C'_2$ is an extremal
doubly even self-dual binary $[24,12,8]$ code.
\end{cor}
\begin{proof}
By Lemmas~\ref{lem:sd} and \ref{lem:sd2}, $C_3$ is self-dual 
while $C'_2$ is doubly even self-dual.
Since $C_3$ has no codeword of weight $3$ by 
Lemma~\ref{lem:norm_m} (ii), 
it suffices to show that $C_3$ has a codeword
of weight $6$ if and only if $C'_2$ has a codeword of 
weight $4$. This follows immediately from Corollary~\ref{cor:C2Cp'}.
\end{proof}

Next, we consider the case $(\ell,m,n)=(4,3,48)$.
\begin{cor}\label{cor:3-4}
Let $H$ be a normalized Hadamard matrix of order $48$, and let
$B$ the binary Hadamard matrix associated to $H$.
Let $C_3$ be the ternary code generated by the rows of $H^T$,
and let
$C'_4$ be the code over $\Z/4\Z$ generated by the rows of $B$.
Then the following statements hold.
\begin{itemize}
\item[{\rm(i)}]
Let $d=2,3$ or $4$.
If $C_3$ has a codeword of weight $3d$,
then $C'_4$ has a codeword of type II norm $8\lceil\frac{d}{2}\rceil$,
and moreover whose nonzero entries are all equal to $\pm1$ when $d=2$ or $3$.
\item[{\rm(ii)}]
Let $d=2$ or $4$.
If $C'_4$ has a codeword $u$ of type II norm $4d$, 
then $C_3$ has a nonzero codeword of weight at most $3d$,
and exactly $3d$ when $\NormI(u)=16$.
\item[{\rm(iii)}]
$C_3$ is an extremal self-dual $[48,24,15]$ code
if and only if $C'_4$ has minimum type II norm $24$.
\end{itemize}
\end{cor}
\begin{proof}
By Lemmas~\ref{lem:sd} and \ref{lem:sd2}, $C_3$ is self-dual while $C'_4$ is type II self-dual.
Since the even norm and the odd norm of a nonzero vector 
over $\Z/3\Z$ of weight $3d$ are
$6\lceil\frac{d}{2}\rceil$ and $6\lfloor\frac{d}{2}\rfloor+3$ respectively,
the even norm is less than $2\lfloor\frac{\ell+2}{2}\rfloor m=18$
for $d\leq 4$
and the odd norm is $9$ for $d=2,3$.
Thus (i) follows from Theorem~\ref{thm:C2Cp'} (i).
(ii) follows from Theorem~\ref{thm:C2Cp'} (ii).
As for (iii), first note that by Lemma~\ref{lem:norm_m}, 
$C_3$ is an extremal self-dual $[48,24,15]$ code
if and only if $C_3$ has no codeword of weight $6,9$ or $12$. By (i) and (ii),
this is equivalent to the non-existence of codewords of type II norm $8$
or $16$ in $C'_4$.
\end{proof}
We will show in the next section that the condition (iii) in 
Corollary~\ref{cor:3-4} is also equivalent to $C'_4$ having
minimum Euclidean norm $24$.

\begin{rem}\label{rem:ext}
It is known that there are at least two inequivalent extremal ternary
self-dual codes of
length $48$, the quadratic residue code and the Pless symmetry code.
The codewords of weight $48$ in these codes constitute the rows and
their negatives of a Hadamard matrix
(\cite[\S 2.8, \S 2.10 of Chap.~3]{SPLAG}).
We will also show in the next section that this is the case for
any extremal ternary self-dual $[48,24,15]$ code.
\end{rem}

\section{Lattices}
We refer the reader to
\cite{SPLAG} for unexplained terminology in lattices.
We write $\Lambda\cong\Lambda'$ if the two lattices $\Lambda$ and $\Lambda'$
are isometric.
Let $C$ be a code of length $n$ over $\Z/m\Z$ with generator
matrix $H$. We regard the entries of $H$ as integers, and
let $\Z^k H$ denote the row $\Z$-module of $H$, that is,
the set of $\Z$-linear combinations of the row vectors of
$H$, where $k$ is the number of rows of $H$. The lattice
$A(C)$ of the code $C$ is defined as
$A(C)=\frac{1}{\sqrt m}\Z^{k+n}\begin{bmatrix}H\\mI\end{bmatrix}$,
and $A(C)$ is integral (resp.\ unimodular, even unimodular)
if and only if $C$ is self-orthogonal (resp.\ self-dual, type II).
If $m$ is odd and $C$ is a self-orthogonal code 
over $\Z_m$, then
$$\min\left(\{\|x\|^2\mid x\in A(C)\setminus\sqrt m\Z^n\}\cap2\Z\right)
=\frac{1}{m}\min_{u\in C\setminus\{0\}}\Norme(u),$$
and thus
\begin{equation}
\min\left(\{\|x\|^2\mid0\neq x\in A(C)\}\cap2\Z\right)
=\min\{2m,\frac{1}{m}\min_{u\in C\setminus\{0\}}\Norme(u)\}.
\label{eq:mine}
\end{equation}
If $C$ is a self-orthogonal code over
$\Z_m$ satisfying $C\subset\allone^\perp$, then
\begin{equation}
\begin{split}
\min\{\|x\|^2\mid0\neq x\in A(C),\;\frac{1}{\sqrt{m}}x\cdot\allone
\equiv0\pmod{2}\}
\\
=\min\{2m,\frac{1}{m}\min_{u\in C\setminus\{0\}}\NormII(u)\},
\end{split}
\label{eq:minII}
\end{equation}

\begin{lem}\label{lem:bin}
Let $H$ be a normalized Hadamard matrix, and let $B=\frac12(H+J)$ 
be the binary
Hadamard matrix associated to $H$. Let $m\geq3$ be an odd integer.
Then $H$ and $B$ generate the same code over $\Z/m\Z$.
\end{lem}
\begin{proof}
The code generated by $H$ can also be generated by $\allone$
and $H+J=2B$. Since $m$ is odd, $B$ and $2B$ generate the same
code over $\Z/m\Z$, while the first row of $B$ is $\allone$ 
since $H$ is normalized. The result follows.
\end{proof}

In the following,
let $m\geq 3$ be an odd integer, $\ell$ an integer such that $(\ell,m)=1$,
$H$ a normalized Hadamard matrix of order $n=4\ell m$.
Then by Lemma~\ref{lem:sd}, 
the code $C_m$ over $\Z/m\Z$ generated by the row vectors of $H^T$
is self-dual, and thus the lattice
$$A(C_m)=\frac{1}{\sqrt m}\Z^{2n}\begin{bmatrix}H^T\\mI\end{bmatrix}$$
is odd unimodular.

Let $h_1=\begin{bmatrix} h_{11}\dots h_{n1}\end{bmatrix}$ 
be the transpose of the first
column of $H$, and let $D=\diag(h_1)$.
Since $H^TD$ is normalized, we have
\[
A(C_m)D=\frac{1}{\sqrt m}\Z^{2n}\begin{bmatrix}\frac{1}{2}(H^TD+J)
\\ mI\end{bmatrix},
\]
by Lemma~\ref{lem:bin}, thus
\begin{equation}\label{eq:ACm}
A(C_m)=\frac{1}{\sqrt m}\Z^{2n}\begin{bmatrix}\frac{1}{2}(H^T+\allone^Th_1)
\\ mI\end{bmatrix}
=\frac{1}{\sqrt m}\Z^{2n+1}\begin{bmatrix}\frac{1}{2}(H^T+\allone^Th_1)
\\ m(I+\allone^Te_1) \\ me_1\end{bmatrix}.
\end{equation}
This implies
\begin{equation}\label{eq:ACmH}
A(C_m)\frac{1}{\sqrt n}H
=\frac{1}{\sqrt \ell}\Z^{2n+1}\begin{bmatrix}\ell(I+\allone^Te_1)
\\ \frac{1}{2}(H+J)\\ \frac{1}{2}\allone \end{bmatrix}.
\end{equation}

So when $\ell=1$, 
the lattice $A(C_m)$ is equivalent to
the unimodular lattice 
\[
D_{n}^+=\Z^{n+1}
\begin{bmatrix}I+\allone^Te_1\\\frac{1}{2}\allone\end{bmatrix}.
\]

For the remainder of this section, we assume $\ell>1$.
By Lemma~\ref{lem:sd2},
the code $C'_\ell$ over $\Z/\ell\Z$ generated by the row vectors of $\frac12(H+J)$
is self-dual, and thus the lattice
\begin{equation}\label{eq:ACl}
A(C'_\ell)=\frac{1}{\sqrt \ell}\Z^{2n}\begin{bmatrix}\frac{1}{2}(H+J)\\\ell I\end{bmatrix}
=\frac{1}{\sqrt \ell}\Z^{2n+1}\begin{bmatrix}\frac{1}{2}(H+J)\\\ell(I+\allone^Te_1)\\\ell e_1\end{bmatrix}
\end{equation}
is unimodular, which is even if and only if $\ell$ is even.

By (\ref{eq:ACm}), the even sublattice of $A(C_m)$ is
\begin{equation}\label{eq:Bm}
B(C_m)=\frac{1}{\sqrt m}\Z^{2n}\begin{bmatrix}\frac{1}{2}(H^T+\allone^Th_1)\\m(I+\allone^Te_1) \end{bmatrix}.
\end{equation}
Analogously, we define a sublattice $B(C'_\ell)$ of $A(C'_\ell)$ as
\begin{equation}\label{eq:Bl}
B(C'_\ell)=\frac{1}{\sqrt \ell}\Z^{2n}\begin{bmatrix}\frac{1}{2}(H+J)\\\ell(I+\allone^Te_1)\end{bmatrix}
\subset A(C'_\ell).
\end{equation}
Then
\begin{equation}\label{eq:B}
B(C_m)=B(C_m)\frac{1}{\sqrt n}H=B(C'_\ell).
\end{equation}
Since
\[
A(C_m)\frac{1}{\sqrt n}H\ni\frac{1}{2\sqrt{\ell}}\allone\notin A(C'_\ell)
\]
by (\ref{eq:ACmH}) and (\ref{eq:ACl}), we see
$A(C_m)\frac{1}{\sqrt n}H\neq A(C'_\ell)$. Thus
there is a unique unimodular lattice containing $B(C'_\ell)$ other than $A(C'_\ell)$ and 
$A(C_m)\frac{1}{\sqrt n}H$
(see \cite{Venkov}).
Let $\Lambda(C'_\ell)$ denote this lattice and let 
\begin{equation}\label{eq:LH}
\Lambda(C_m)=\Lambda(C'_\ell)\frac{1}{\sqrt n}H^T
\end{equation}
be the unique unimodular
lattice containing $B(C_m)$
other than $A(C_m)$ and $A(C'_\ell)\frac{1}{\sqrt n}H^T$.
The relationship between the lattices introduced so far can
conveniently described by the following diagram, where
a line denotes inclusion.

\setlength{\unitlength}{.15mm}
\begin{center}
\begin{picture}(200,200)(60,-50)
\put(0,0){\line(0,1){100}}
\put(0,0){\line(1,1){100}}
\put(0,0){\line(-1,1){100}}
\put(-100,120){\makebox(0,0){$A(C_m)$}}
\put(0,120){\makebox(0,0){$\Lambda(C_m)$}}
\put(120,120){\makebox(0,0){$A(C'_\ell)\frac{1}{\sqrt{n}}H^T$}}
\put(0,-20){\makebox(0,0){$B(C_m)$}}
\put(150,50){\vector(1,0){100}}
\put(200,0){\makebox(0,0){$\frac{1}{\sqrt{n}}H$}}
\end{picture}
\begin{picture}(100,100)(-160,-50)
\put(0,0){\line(0,1){100}}
\put(0,0){\line(1,1){100}}
\put(0,0){\line(-1,1){100}}
\put(-120,120){\makebox(0,0){$A(C_m)\frac{1}{\sqrt{n}}H$}}
\put(0,120){\makebox(0,0){$\Lambda(C'_\ell)$}}
\put(100,120){\makebox(0,0){$A(C'_\ell)$}}
\put(0,-20){\makebox(0,0){$B(C'_\ell)$}}
\end{picture}
\end{center}

Since $A(C'_\ell)=B(C'_\ell)\cup (B(C'_\ell)+\sqrt{\ell}e_1)$ by (\ref{eq:ACl}) and
$A(C_m)\frac{1}{\sqrt n}H=B(C'_\ell)\cup (B(C'_\ell)+\frac{1}{2\sqrt \ell}\allone)$ by (\ref{eq:ACmH}),
we have
\begin{align}
\label{eq:L'}
\Lambda(C'_\ell)&=\frac{1}{\sqrt \ell}\Z^{2n+1}
\begin{bmatrix}\frac{1}{2}(H+J)\\\ell(I+\allone^Te_1)\\
\ell e_1+\frac{1}{2}\allone\end{bmatrix},\\
\label{eq:L}
\Lambda(C_m)&=\frac{1}{\sqrt m}\Z^{2n+1}
\begin{bmatrix}\frac{1}{2}(H^T+\allone^Th_1)\\m(I+\allone^Te_1)
\\me_1+\frac{1}{2}h_1\end{bmatrix}.
\end{align}
Observe also, by (\ref{eq:ACl}),
\begin{equation}\label{eq:Al}
A(C'_\ell)\frac{1}{\sqrt n}H^T
=\frac{1}{\sqrt{m}}\Z^{2n+1}
\begin{bmatrix}m(I+\allone^Te_1)\\ \frac{1}{2}(H^T+\allone^Th_1)\\
\frac{1}{2}h_1\end{bmatrix}.
\end{equation}
Then, 
\begin{equation}
A(C'_\ell)\frac{1}{\sqrt n}H^T\setminus B(C_m)=B(C_m)-\frac{1}{2\sqrt m}h_1
\subset\frac{1}{2\sqrt m}(1+2\Z)^n\label{eq:A-Bm}
\end{equation}
by (\ref{eq:Bm}) and (\ref{eq:Al}),
\begin{equation}
\Lambda(C_m)\setminus B(C_m)=B(C_m)-\frac{1}{\sqrt m}(me_1+\frac{1}{2}h_1)
\subset\frac{1}{2\sqrt m}(1+2\Z)^n\label{eq:L-Bm}
\end{equation}
by (\ref{eq:Bm}) and (\ref{eq:L}),
\begin{equation}
A(C_m)\frac{1}{\sqrt n}H \setminus  B(C'_\ell) =B(C'_\ell)-\frac{1}{2\sqrt \ell}\allone
\subset\frac{1}{2\sqrt \ell}(1+2\Z)^n\label{eq:A-Bl}
\end{equation}
by (\ref{eq:ACmH}) and (\ref{eq:Bl}), and
\begin{equation}
\Lambda(C'_\ell) \setminus  B(C'_\ell) =B(C'_\ell)-\frac{1}{\sqrt \ell}(\ell e_1+\frac{1}{2}\allone)
\subset\frac{1}{2\sqrt \ell}(1+2\Z)^n\label{eq:L-Bl}
\end{equation}
by (\ref{eq:Bl}) and (\ref{eq:L'}).

\begin{thm}\label{thm:14}
Let $H$ be a normalized Hadamard matrix of order $n=4\ell m$,
where $m\geq 3$ is an odd integer, $\ell\geq 2$ an integer such that $(\ell,m)=1$.
For an even integer $d<\min\{2\ell,2m\}$,
the following statements {\rm(i)--(iii)} are equivalent,
and moreover if $d\leq\max\{\ell,m+\delta_{\ell\bmod2,0}\}$, 
{\rm(iii) and (iv)} are equivalent:
\begin{enumerate}
\item $C_m$ has minimum even norm $dm$,
\item $C'_\ell$ has minimum type II norm $d\ell$,
\item $B(C_m)$ has minimum norm $d$,
\item $\Lambda(C_m)$ has minimum norm $d$.
\end{enumerate}
\end{thm}
\begin{proof}
The relations between $A(C_m)$, $A(C'_\ell)$
and $B(C_m)$, $B(C'_\ell)$ 
are given as
\begin{align*}
B(C_m)&=\{x\in A(C_m)\mid\|x\|^2\equiv 0\pmod 2\},\\
B(C'_\ell)&=\{x\in A(C'_\ell)\mid\frac{1}{\sqrt \ell}x\cdot\allone\equiv 0\pmod 2\}.
\end{align*}
Thus, by (\ref{eq:mine}), (\ref{eq:minII}) and (\ref{eq:B}),
we have
\begin{align}
\label{eq:minB}
\min B(C_m)&=\min\{2m,\frac{1}{m}\min_{u\in C_m\setminus\{0\}}\Norme(u)\}\nonumber\\
=\min B(C'_\ell)&=\min\{2\ell,\frac{1}{\ell}\min_{u\in C'_\ell\setminus\{0\}}\NormII(u)\}.
\end{align}
Since $d<\min\{2\ell,2m\}$, 
the equivalence of (i)--(iii) is established.

We have $\min(\Lambda(C_m)\setminus B(C_m))\geq \ell$ by (\ref{eq:L-Bm})
and $\min(\Lambda(C'_\ell)\setminus B(C'_\ell))\geq m$ by (\ref{eq:L-Bl}),
and thus $\min(\Lambda(C_m)\setminus B(C_m))\geq\max\{\ell,m\}$.
If $\ell\equiv0\pmod2$, $\Lambda(C_m)$ is an even lattice, so
$\min(\Lambda(C_m)\setminus B(C_m))\geq\max\{\ell,m+1\}$.
This shows the equivalence of (iii) and (iv).
\end{proof}

We have the following
as the complement of above theorem.
\begin{cor}\label{cor:15}
Let $H$ be a normalized Hadamard matrix of order $n=4\ell m$,
where $m\geq 3$ is an odd integer, $\ell\geq 2$ an integer such that $(\ell,m)=1$.
Let $d=\min\{2\ell,2m\}$.
Then the following statements {\rm(i)--(iii)} are equivalent,
and moreover if $d\leq\max\{\ell,m+\delta_{\ell\bmod2,0}\}$, 
{\rm(iii)} and {\rm(iv)} are equivalent:
\begin{enumerate}
\item $C_m$ has minimum even norm at least (exactly if $d=2\ell$) $dm$,
\item $C'_\ell$ has minimum type II norm at least (exactly if $d=2m$) $d\ell$,
\item $B(C_m)$ has minimum norm $d$,
\item $\Lambda(C_m)$ has minimum norm $d$.
\end{enumerate}
\end{cor}
\begin{proof}
By Theorem~\ref{thm:14}, we see the equivalence of (i)$^\prime$ $C_m$ has minimum even norm less than $dm$,
(ii)$^\prime$ $C'_\ell$ has minimum type II norm less than $d\ell$, and
(iii)$^\prime$ $B(C_m)$ has minimum norm less than $d$.
Since $\min B(C_m)$ is at most $d$ by (\ref{eq:minB}),
(i)--(iii) are the negatives of (i)$^\prime$--(iii)$^\prime$ respectively.
Exactness in (i) and (ii) follows from
\[
\min B(C_m)=
\begin{cases}
\displaystyle
\frac{1}{m}\min_{u\in C_m\setminus\{0\}}\Norme(u) 
&\text{if $d=2\ell$,}\\
\displaystyle
\frac{1}{\ell}\min_{u\in C'_\ell\setminus\{0\}}\NormII(u) 
&\text{if $d=2m$.}
\end{cases}
\]
The equivalence of (iii) and (iv) follows also from
Theorem~\ref{thm:14} provided
$d\leq\max\{\ell,m+\delta_{\ell\bmod2,0}\}$.
\end{proof}
Note that the minimum norm of $C_m$ and $C'_\ell$ are both at most $n/2=2\ell m$,
given by the sum of two distinct rows of $H^T$ for $C_m$,
and by any row except the first one of $B$ for $C'_\ell$.

Setting $(\ell,m,n)=(2,3,24)$ in Corollary~\ref{cor:15}, we have another proof of Corollary~\ref{cor:1}.
\begin{cor}
Let $H$ be a normalized Hadamard matrix of order $24$.
The following statements are equivalent:
\begin{enumerate}
\item $C_3$ has minimum weight $9$,
\item $C'_2$ has minimum weight $8$,
\item $\Lambda(C_3)$ has minimum norm $4$ (hence is isometric 
to the Leech lattice).
\end{enumerate}
\end{cor}
\begin{proof}
Since $C_3$ has minimum even norm $12$ if and only if
it has minimum weight $9$ by the extremality condition,
the result follows from Corollary~\ref{cor:15}.
We note that when (iii) occurs, $\Lambda(C_3)$
is isometric to the Leech lattice by \cite[chap.~12]{SPLAG}.
\end{proof}

Let $k$ be even and let $H$ be a skew Hadamard matrix of order $n=4k-4$ with all diagonal entries $-1$.
In \cite{Mc}, McKay gives a even unimodular lattice
$$L=\frac{1}{\sqrt k}\Z^{2n}\begin{bmatrix}I_n&H-I_n\\O&kI\end{bmatrix},$$
and asserts that its minimum norm is $4$ for $k\geq 4$ (See also \cite{Chap}).
Without loss of generality,
we may assume the first row of $H$ to be $-\allone$.
Then
$$\tilde{H}=\begin{bmatrix}-H&H^TD\\H&H^TD\end{bmatrix},$$ where $D=\diag(-1,1,1,\dots,1)$,
is a normalized Hadamard matrix. We describe an isometry from $L$ to $\Lambda(C'_2)$,
where $C'_2$ is the binary doubly even self-dual code obtained from the binary Hadamard
matrix associated to $\tilde{H}$. Set
$$U=\frac{1}{2\sqrt k}\begin{bmatrix}I_n&-H+I_n\\H^T-I&I_n\end{bmatrix}.$$
Since $U$ is an orthogonal matrix and
$$LU=\frac{1}{2}\Z^{2n}\begin{bmatrix}4I_n&O\\H^T-I_n&I_n\end{bmatrix},$$
$L$ is isometric to the lattice obtained from the $\Z_4$-code
with generator matrix $\begin{bmatrix}H^T-I_n&I_n\end{bmatrix}$.
Furthermore, set
$$V=\frac{1}{\sqrt 2}\begin{bmatrix}I_n&D\\-I_n&D\end{bmatrix}\text{, and}$$
$$M=\frac{1}{2\sqrt 2}\begin{bmatrix}4I_n&4D\\H^T-2I_n&H^TD\end{bmatrix}.$$
Then
\begin{align*}
\Lambda(C'_2)&=\frac{1}{\sqrt 2}\Z^{2n+1}
\begin{bmatrix}\frac{1}{2}(\tilde{H}+J)\\2(I+\allone^Te_1)\\
2e_1+\frac{1}{2}\allone\end{bmatrix}\\
&=\Z^{2n+1}\begin{bmatrix}\frac{1}{2}(H^T+J_n)-\allone^Te_1&-I_n-\allone^Te_1\\
\frac{1}{2}(H^T+J_n)-\allone^Te_1+\frac{n}{4}I_n&-I_n-\allone^Te_1-H\\
\frac{1}{2}(H^T+J_n)-\allone^Te_1&-2(I_n+\allone^Te_1)\\
\frac{1}{2}(J_nD-DH^T)+D&2(I_n+\allone^Te_1)D\\
\allone-2e_1 & -3e_1\end{bmatrix}M\\
&\subset\Z^{2n}M=\frac{1}{2}\Z^{2n}\begin{bmatrix}4I_n&O\\H^T-I_n&I_n\end{bmatrix}V=LUV.
\end{align*}
Since $L$ and $\Lambda(C'_2)$ are both unimodular and $V$ is an orthogonal matrix,
we conclude that $L$ is isometric to $\Lambda(C'_2)$.

\begin{cor}\label{cor:48-1}
Let $H$ be a normalized Hadamard matrix of order $48$.
The following statements are equivalent:
\begin{enumerate}
\item $C_3$ has minimum weight $15$,
\item $C'_4$ has minimum type II norm $24$,
\item $B(C_3)$ has minimum norm $6$.
\end{enumerate}
\end{cor}
\begin{proof}
Since the minimum weight of $C_3$ is at most $15$ by the extremality condition,
the result follows by setting $(\ell,m)=(4,3)$ in
Corollary~\ref{cor:15}.
\end{proof}

As a matter of fact, we have stronger result by the following argument.
\begin{lem}\label{lem:num}
Let $H$ be a normalized Hadamard matrix of order $n=4\ell m$,
where $m\geq 3$ is an odd integer, $\ell\geq 2$ an integer such that $(\ell,m)=1$,
and assume $H^T$ is also normalized.
Then the number of norm $\ell$ vectors of $A(C'_\ell)\setminus B(C'_\ell)$
(resp. $\Lambda(C_m)\setminus B(C_m)$) is equal to the number of codewords of $C_m$
of even (resp. odd) weight whose nonzero entries are all equal to $1$.
\end{lem}
\begin{proof}
Set
\begin{align*}
L&=A(C_m)-\frac{1}{2\sqrt{m}}\allone,\\
X&=\{x\in L\mid \|x\|^2=\ell\}.
\end{align*}
Then every element of $X$
is of the form $v=\frac{1}{2\sqrt{m}}(\pm1,\dots,\pm1)$,
and hence the map
\begin{align*}
\rho:X&\rightarrow C_m\\
v&\mapsto \left(\sqrt{m}v+\frac{1}{2}\allone\right)\bmod m,
\end{align*}
gives a one-to-one correspondence between $X$ and
the set of codewords of $C_m$ whose nonzero entries are all equal to $1$.
For $v\in X$, we have
$\wt(\rho(v))=\Norm(\rho(v))=m\|v+\frac{1}{2\sqrt m}\allone\|^2$.
Thus $\wt(\rho(v))$ is even if and only if
$v+\frac{1}{2\sqrt m}\allone \in B(C_m)$.
Since $H^T$ is normalized, 
(\ref{eq:B}) and (\ref{eq:A-Bm}) imply
$(A(C'_\ell)\setminus B(C'_\ell))\frac{1}{\sqrt n}H^T=
B(C_m)-\frac{1}{2\sqrt{m}}\allone$,
and hence the set of norm $\ell$ vectors of
$A(C'_\ell)\setminus B(C'_\ell)$ is
\[
\left\{\frac{1}{\sqrt n}vH\mid v\in X,\;
\wt(\rho(v))\text{ even}\right\}.
\]
Similarly, since
\[
L=(B(C_m)-\frac{1}{2\sqrt m}\allone)\cup
(B(C_m)-\frac{1}{\sqrt m}(me_1+\frac{1}{2}\allone))
\quad\text{(disjoint),}
\]
and by (\ref{eq:L-Bm}),
the set of norm $\ell$ vectors of 
$\Lambda(C_m)\setminus B(C_m)$
is $\{v\in X\!\mid\! \wt(\rho(v))\!\text{ odd}\}$.
\end{proof}


A ternary self-dual code of length $n$ has minimum weight 
at most $3\lfloor n/12\rfloor +3$ (see \cite{MSic}),
thus at most $15$ for $n=48$.
A type II self-dual code over $\Z/4\Z$ of length $n$ has
minimum Euclidean norm
at most $8\lfloor n/24\rfloor+8$ (see \cite[Corollary 13]{BSBM}),
thus at most $24$ for $n=48$. 
An $n$-dimensional even unimodular lattice has minimum norm at most
$2\lfloor n/24\rfloor+2$, thus at most $6$ for $n=48$.
A code or a lattice achieving the upper bound is called extremal.

By \cite[Proposition 3.3]{Gab}, the complete weight enumerator of
any extremal $[48,24,15]$ ternary self-dual code with all-one vector
is uniquely determined to
\begin{equation}\label{eq:cwe}
W(x,y,z)=\sum x^{48}+94\sum x^{24}y^{24}+x^3y^3z^3(\dots),
\end{equation}
given in \cite[Table 1]{Koch},
where the sums are to be taken over the cyclic permutations of $x,y,z$.
Now we have the following sharpening of Corollary~\ref{cor:3-4}(iii)
and Corollary~\ref{cor:48-1}.

\begin{thm}
Let $H$ be a normalized Hadamard matrix of order $48$, and let
$B$ be the binary Hadamard matrix associated to $H$.
Let $C_3$ be the ternary code generated by the rows of $H^T$, and let
$C'_4$ be the code over $\Z/4\Z$ generated by the rows of $B$.
The following statements are equivalent:
\begin{enumerate}
\item $C_3$ is extremal,
\item $C'_4$ is extremal,
\item $\Lambda(C_3)$ is extremal.
\end{enumerate}
\end{thm}
\begin{proof}
Since any row of $B$ except the 
first one gives a codeword of $C'_4$ with type II norm $24$,
(ii) implies that $C'_4$ has minimum type II norm $24$. 
Thus (ii)$\Rightarrow$(i) follows from Corollary~\ref{cor:48-1}.
If $\Lambda(C_3)$ has minimum norm $6$, then by (\ref{eq:minB}),
$B(C_3)$ has minimum norm $6$.
Thus (iii)$\Rightarrow$(i) follows also
from Corollary~\ref{cor:48-1}.

To prove (i)$\Rightarrow$(ii), suppose that
$C_3$ has minimum weight $15$. 
Let $D=\diag(h_1)$ where $h_1$ is the first row of $H^T$.
Then $H'=DH$ is a normalized Hadamard matrix such that ${H'}^T$ is also
normalized.
The rows of $\frac{1}{2}(H'+J)$ generate the $\Z/4\Z$ code $C'_4$ since
$\frac{1}{2}D(H'+J)+\frac{1}{2}(\allone-h_1)^T\allone=B$, 
while the rows of ${H'}^T$ generate the ternary code $C_3D$
which is equivalent to $C_3$,
and (\ref{eq:LH}) implies $\Lambda(C_3D)\cong\Lambda(C_3)$.
Thus, we may assume from the beginning that both $H$ and $H^T$ are
normalized. Then Lemma~\ref{lem:num} implies 
\begin{equation}\label{eq:602}
\left|\{v\in A(C'_4)\setminus B(C'_4)\mid \|v\|^2=4\}\right|\\
=\left|\{x\in C_3\cap \{0,1\}^{48}\mid \wt(x)\text{ even}\}\right|.
\end{equation}
Note that $B(C'_4)\cong B(C_3)$
has no vector of norm $2$ or $4$
by (\ref{eq:B}) and Corollary~\ref{cor:48-1}, 
and $A(C'_4)\setminus B(C'_4)$ has no vector of norm $2$
by (\ref{eq:B}) and (\ref{eq:A-Bm}). 
Thus, the left-hand side of
(\ref{eq:602}) coincides with the number of norm $4$ vectors in $A(C'_4)$.
On the other hand, as $H^T$ is normalized, $C_3$ contains the all-one vector,
hence the right-hand side of (\ref{eq:602}) equals $1+94+1=96$ by (\ref{eq:cwe}).
It follows that 
the $96$ norm $4$ vectors of $A(C'_4)$ are $\pm2e_i$ $(i=1,\dots,48)$,
and thus $C'_4$ has no codeword of Euclidean norm $16$.
Therefore, $C'_4$ has minimum Euclidean norm at least $24$, and
hence equal to $24$. This proves (i)$\Rightarrow$(ii).

Replacing $A(C'_4)$ by $\Lambda(C_3)$, $B(C'_4)$ by $B(C_3)$ and
$\wt(x)$ even by $\wt(x)$ odd in the proof
of (i)$\Rightarrow$(ii), and by (\ref{eq:L-Bm}), we have that
$\Lambda(C_3)$ has no vector of norm $2$ or $4$. Since $\Lambda(C_3)$ is even,
it has minimum norm $6$. This proves (i)$\Rightarrow$(iii).
\end{proof}

As mentioned in Remark~\ref{rem:ext}, there are at least two extremal
ternary self-dual $[48,24,15]$ codes, namely,
the quadratic residue code $C_{48q}^{(3)}$
and the Pless symmetry code $C_{48p}^{(3)}$.
The code $C_{48q}^{(3)}$ (resp. $C_{48p}^{(3)}$) corresponds to
the extremal $\Z_4$-code $C_{48q}^{(4)}$ (resp. $C_{48p}^{(4)}$),
and the extremal even unimodular lattice $P_{48q}$ (resp. $P_{48p}$) \cite{HKMV}.
There is another known extremal even unimodular lattice $P_{48n}$ \cite{Nebe}.
But it is not known whether $P_{48n}$ has a corresponding extremal ternary code.

The following is an analogue of \cite[Theorem 5]{LPS}.
\begin{thm}
Every extremal ternary self-dual code
of length $48$ is generated by a Hadamard matrix.
\end{thm}
\begin{proof}
Without loss of generality, we may assume $\allone\in C$.
Then (\ref{eq:cwe}) implies that $C$ is admissible in the sense of \cite{Koch}.
As remarked at the end of the paper \cite{Koch}, it follows from
\cite[Proposition 2]{Koch} that the $96$ codewords of weight $48$ in $C$
constitute the rows and their negatives of a Hadamard matrix.
The result then follows from Lemma~\ref{lem:sd}.
\end{proof}

\noindent
{\bf Acknowledgements}\\
We would like to thank Masaaki Kitazume for bringing this
problem to the authors' attention.
We would also like to
thank Masaaki Harada for helpful discussions.
Finally, we would like to thank John McKay for pointing
out a connection to his construction of the Leech lattice
from a Hadamard matrix of order $12$ (\cite{Mc}).


\begin{thebibliography}{9}
\bibitem{AK}
E. F. Assmus, Jr. and J. D. Key, ``Designs and Their Codes,'' Cambridge University Press, Cambridge, 1992.
\bibitem{BDHO}
E. Bannai, S. T. Dougherty, M. Harada and M. Oura,
Type II codes, even unimodular lattices, and invariant rings,
IEEE Trans. Inform. Theory 45 (1999), 1194--1205.
\bibitem{BSBM}
A. Bonnecaze, P. Sol\'{e}, C. Bachoc and B. Mourrain, 
Type II Codes over $\Z_4$,
IEEE Trans. Inform. Theory 43 (1997), 969--976.
\bibitem{Chap}
R. Chapman,
Double circulant constructions of the Leech lattice,
J. Austral. Math. Soc. (Series A), 69 (2000), 287--297.
\bibitem{SPLAG} J. H. Conway and N. J. A. Sloane,
``Sphere Packing, Lattices and Groups,'' 3rd ed.,
Springer-Verlag, New York, 1999.
\bibitem{Gab}
P. Gaborit,
Construction of new extremal unimodular lattices,
European J. Combin., 25 (2004), 549--564.
\bibitem{Hall}
M. Hall, Jr., ``Combinatorial Theory,'' 
2nd edition, Wiley, New York, 1986.
\bibitem{HKMV}
M. Harada, M. Kitazume, A. Munemasa, B. Venkov,
On some self-dual codes and unimodular lattices in dimension $48$,
European J. Combin., 26 (2005), 543--557.
\bibitem{ILu}
N. Ito, J. S. Leon and J. Q. Longyear, 
The $24$-dimensional Hadamard matrices and their automorphism
groups, unpublished.
\bibitem{IL}
N. Ito, J. S. Leon and J. Q. Longyear, 
Classification of $3$-$(24,12,5)$ designs and $24$-dimensional
Hadamard matrices,
J. Combin. Theory, Ser. A, 31 (1981), 66--93.
\bibitem{K} H. Kimura,
New Hadamard matrix of order $24$,
Graphs Combin. 5 (1989), 235--242.
\bibitem{Koch} H. Koch,
The $48$-dimensional analogues of the Leech lattice,
Proc. Steklov Inst. Math. 208 (1995), 172--178.
\bibitem{LPS} J. S. Leon, V. Pless and N. J. A. Sloane,
On ternary self-dual codes of length $24$, IEEE Trans.
Inform. Theory 27 (1981), 176--180.
\bibitem{Mc} J. McKay,
A setting for the Leech lattice,
in Finite Groups '72, North-Holland, Amsterdam, 1973, 117--118.
\bibitem{MS} F. J. MacWilliams and N. J. A. Sloane,
{``The Theory of Error-Correcting Codes,''}
North-Holland, Amsterdam, 1977.
\bibitem{MSic} C. L. Mallows and N. J. A. Sloane,
An upper bound for self-dual codes, Inform. Control
22 (1973), 188--200.
\bibitem{MS2} C. L. Mallows and N. J. A. Sloane,
Weight enumerators of self-orthogonal codes over $GF(3)$,
SIAM J. Algebr. Discr. Meth. 2 (1981), 452--480.
\bibitem{Nebe} G. Nebe,
Some cyclo-quaternionic lattices, J. Alg. 199 (1998), 472--498.
\bibitem{Newman} M. Newman, ``Integral matrices,''
Pure and Applied Mathematics 45,
Academic Press, New York, 1972.
\bibitem{Venkov} B. B. Venkov, Odd unimodular lattices,
J. Math. Sci. 17 (1986), 1967--1974.
\bibitem{WSW} W. D. Wallis, A. P. Street and J. S. Wallis,
``Combinatorics: Room Squares, Sum-Free Sets, Hadamard Matrices,''
Lecture Notes in Mathematics 292, Springer-Verlag, Berlin, 1972.
\end{thebibliography}
\end{document}